\documentclass[a4paper]{amsart}

\usepackage{amsrefs}
\usepackage{amssymb}
\usepackage[cmtip,all]{xy}

\newtheorem{theorem}{Theorem}[section]
\newtheorem{lemma}[theorem]{Lemma}
\newtheorem{proposition}[theorem]{Proposition }

\theoremstyle{remark}
\newtheorem{remark}[theorem]{Remark}

\numberwithin{equation}{section}

\newcommand{\be}{{\boldsymbol{e}}}

\newcommand{\cC}{\mathcal{C}}

\newcommand{\EE}{\mathbb{E}}
\newcommand{\ev}{\mathrm{ev}}

\newcommand{\Mbar}{\overline{\mathcal{M}}}
\newcommand{\Mfr}{\mathfrak{M}}

\newcommand{\cO}{\mathcal{O}}
\newcommand{\PP}{\mathbb{P}}
\newcommand{\cV}{\mathcal{V}}
\DeclareMathOperator{\vdim}{vdim}
\newcommand{\vir}{\text{\rm vir}}

\newcommand{\ZZ}{\mathbb{Z}}
\begin{document}

\title[The Quantum Lefschetz Hyperplane Principle can fail] 
{The Quantum Lefschetz Hyperplane Principle can fail for positive orbifold
  hypersurfaces}

\author[Coates]{Tom Coates}
\address{Department of Mathematics\\
Imperial College London\\
180 Queen's Gate\\
London SW7 2AZ 
\\UK}
\email{t.coates@imperial.ac.uk}

% \author[Corti]{Alessio Corti}
% \address{Department of Mathematics\\
% Imperial College London\\
% 180 Queen's Gate\\
% London SW7 2AZ 
% \\UK}
% \email{a.corti@imperial.ac.uk}

\author[Gholampour]{Amin Gholampour}
\address{Department of Mathematics\\
University of Maryland\\
1301 Mathematics Building\\
College Park, MD 20742-4015\\
USA}
\email{amingh@math.umd.edu}

\author[Iritani]{Hiroshi Iritani}
\address{Department of Mathematics\\
Graduate School of Science\\
Kyoto University\\
Oiwake-cho\\
Kitashirakawa\\
Sakyo-ku\\
Kyoto, 606-8502\\
Japan}
\email{iritani@math.kyoto-u.ac.jp}

\author[Jiang]{Yunfeng Jiang}
\address{Department of Mathematics\\
Imperial College London\\
180 Queen's Gate\\
London SW7 2AZ 
\\UK}
\email{y.jiang@imperial.ac.uk}

\author[Johnson]{Paul Johnson}
\address{Mathematics Department\\
Columbia University\\
Room 509, MC 4406\\
2990 Broadway\\
New York, NY 10027\\
USA} 
\email{paul.da.johnson@gmail.com}

\author[Manolache]{Cristina Manolache}
\address{Institut f\"ur Mathematik\\ 
Humboldt Universit\"at\\
Berlin\\
Germany}
\email{manolach@mathematik.hu-berlin.de}

\thanks{The authors thank Huai-Liang Chang, Alessio Corti, Behrang Noohi, and Michael
  Rose for useful conversations. This research is supported by a Royal
  Society University Research Fellowship (TC), European Research
  Council Starting Grant 240123-GWT, an NSERC Postdoctoral Fellowship
  (AG), and an NSF Postdoctoral Fellowship DMS-0902754 (PJ)}
  


\date{}

%    Abstract is required.
\begin{abstract}
  We show that the Quantum Lefschetz Hyperplane Principle can fail for
  certain orbifold hypersurfaces and complete intersections. It can
  fail even for orbifold hypersurfaces defined by a section of an
  ample line bundle.
\end{abstract}

\maketitle

\section{Introduction}

Let $X$ be a projective algebraic variety.  Let $g$ and $n$ be
non-negative integers, $d$ be an element of $H_2(X;\ZZ)$, and
$X_{g,n,d}$ be the moduli stack of degree-$d$ stable maps to $X$ from
genus-$g$ curves with $n$ marked points \cite{Kontsevich--Manin}.
Gromov--Witten invariants of $X$ are intersection numbers in
$X_{g,n,d}$ against the virtual fundamental cycle $[X_{g,n,d}]^\vir$
\citelist{\cite{Behrend--Fantechi} \cite{Li--Tian}}.  Let $Y \subset
X$ be a complete intersection cut out by a section of a vector bundle
$E \to X$ which is the direct sum of line bundles $E = \oplus E_j$.
The inclusion $i:Y \to X$ induces a morphism of moduli stacks
$\iota:Y_{g,n,\delta} \to X_{g,n,i_\star \delta}$.  Suppose that the line bundles $E_j$
each satisfy the positivity condition:
\begin{equation*}
  \tag{$\ast$}
  \text{$c_1(E_j) \cdot d \geq 0$ whenever $d$ is the degree
    of genus-zero stable map to $X$}
\end{equation*}
Then:
\[
\tag{$\dagger$}
\sum_{\delta : i_\star \delta = d}
\iota_\star [Y_{0,n,\delta}]^\vir = [X_{0,n,d}]^\vir \cap \be(E_{0,n,d})
\]
where $\be$ is the Euler class and $E_{0,n,d}$ is a certain vector
bundle on $X_{0,n,d}$, described in \S\ref{sec:convexity} below.
Equality ($\dagger$) lies at the heart of the Quantum Lefschetz
Hyperplane Principle, and hence of the proof of mirror symmetry for
toric complete intersections \citelist{\cite{Givental:equivariant}
  \cite{Givental:toric} \cite{LLY:1} \cite{LLY:2} \cite{LLY:3}}.  (See
\cite{Cox--Katz--Lee} for a very clear explanation of this.)

In this paper we show by means of examples that, for orbifold complete
intersections, ($\ast$) does not imply ($\dagger$).  We give examples
of smooth orbifolds $X$ and complete intersections $Y \subset X$ cut
out by sections of vector bundles $E = \oplus E_j \to X$ such that
each $E_j$ is a line bundle that satisfies ($\ast$) but there is no
cohomology class $e$ on $X_{0,n,d}$ with:
\[
\sum_{\delta : i_\star \delta = d}
\iota_\star [Y_{0,n,\delta}]^\vir = [X_{0,n,d}]^\vir \cap e
\]
In particular there is no vector bundle $E_{0,n,d}$ on $X_{0,n,d}$
such that ($\dagger$) holds.  Thus the Quantum Lefschetz Hyperplane
Principle, as currently understood, can fail for positive orbifold
complete intersections.

\begin{remark}
  There does not seem to be a universally-accepted definition of
  ampleness for line bundles on orbifolds (see
  \cite{Ross--Thomas}*{\S2.5} for one possibility) but any reasonable
  definition will imply property ($\ast$).
\end{remark}

\section{Genus-one Gromov--Witten invariants of the quintic 3-fold}

It is well-known that there is no straightforward analog of the
Quantum Lefschetz Hyperplane Principle for higher-genus stable maps
(those with $g>0$), even when both $X$ and $Y$ are smooth
varieties. Givental gave an example that demonstrates this
\cite{Givental:elliptic}; for the convenience of the reader we repeat
his argument here.  Recall that for any smooth projective variety $X$,
the moduli stack $X_{1,1,0}$ of degree-zero one-pointed stable maps
from genus-one curves to $X$ is isomorphic to $X \times \Mbar_{1,1}$.
Let $\pi_1:X_{1,1,0} \to X$ denote projection to the first factor,
$\pi_2:X_{1,1,0} \to \Mbar_{1,1}$ denote projection to the second
factor, and $\psi_1 \in H^2(\Mbar_{1,1})$ denote the universal
cotangent line class.  The virtual fundamental class is:
\[
\big[X_{1,1,0}\big]^\vir = \big[X_{1,1,0}\big] \cap
\Big(
\pi_1^\star (c_{D}(TX)) - \pi_1^\star (c_
{D-1}(TX)) \cup \pi_2^\star (\psi_1)
\Big)
\]
where $D$ is the complex dimension of $X$.  

Now let $X = \PP^4$ and $Y \subset X$ be a quintic threefold, i.e.~$Y$
is the hypersurface cut out by a generic section of $E=\cO(5) \to X$.
As before, let $i:Y \to X$ be the inclusion and $\iota:Y_{1,1,0} \to
X_{1,1,0}$ be the induced map of moduli stacks.  We will show that
there is no cohomology class $e$ on $X_{1,1,0}$ such that:
\[
\iota_\star [Y_{1,1,0}]^\vir = [X_{1,1,0}]^\vir \cap e
\]
Since both $[Y_{1,1,0}]^\vir$ and $[X_{1,1,0}]^\vir$ have the same
dimension, this amounts to showing that $\iota_\star [Y_{1,1,0}]^\vir$
is not a scalar multiple of $[X_{1,1,0}]^\vir$.

Let $h \in H^2(X)$ denote the first Chern class of the line bundle
$\cO(1)$ on $X$.  Applying the total Chern class to both sides of the
equality:
\[
TY \oplus i^\star \cO(5) = i^\star TX
\]
yields $c_1(TY) = 0$, $c_2(TY) = 10 i^\star (h^2)$, $c_3(TY) =
-40i^\star (h^3)$.  Thus:
\begin{align*}
  \iota_\star \big[Y_{1,1,0}\big]^\vir &= \iota_\star \big( {-40 }
  i^\star (h^3) - 10 i^\star (h^2)   \psi_1 \big)
 \\ &= \big( {-40 } h^3 - 10 h^2   \psi_1 \big) \cup \iota_\star 1 
 \\ &= -200 h^4 - 50 h^3 \psi_1
\end{align*}
where in the second line we used the projection formula and in the
last line we used the fact that the normal bundle to the inclusion $\iota$ is
$\pi_1^\star \cO(5)$.  On the other hand:
\begin{align*}
  \big[X_{1,1,0}\big]^\vir &= 5h^4 - 10h^3 \psi_1
\end{align*}
and so $\iota_\star [Y_{1,1,0}]^\vir$ is not a scalar multiple of
$[X_{1,1,0}]^\vir$.

\section{A trivial example}

Let $X$ be the orbifold $\PP(1,1,2,2)$, and let $Y = \PP(1,2,2)$ be
the orbifold hypersurface in $X$ defined by the vanishing of a section
of $\cO(1)$.  Let $X_{0,\vec{4},0}$ and $Y_{0,\vec{4},0}$
denote\footnote{The vector $\vec{4}$ in the subscript here is to
  emphasize the fact that we specify not only the number of marked
  points on the curves but also the isotropy group at each marked
  point.} the moduli stacks of genus-zero degree-zero stable maps to
(respectively) $X$ and $Y$, from orbicurves with four marked points
such that the isotropy group at each marked point is $\mu_2$.  As
before, write $i:Y \to X$ for the inclusion map, and
$\iota:Y_{0,\vec{4},0} \to X_{0,\vec{4},0}$ for the induced morphism
of moduli stacks.  We have $\vdim X_{0,\vec{4},0} = 0$ and $\vdim
Y_{0,\vec{4},0} = 1$, so for dimensional reasons there is no
cohomology class $e$ on $X_{0,\vec{4},0}$ such that:
\[
\iota_\star \big[Y_{0,\vec{4},0}\big]^\vir = \big[X_{0,\vec{4},0}\big]^\vir \cap
e
\]

% \begin{remark}
%   One might consider this to be a ``fake counterexample'', because in
%   fact both $X_{0,\vec{4},0}$ and $Y_{0,\vec{4},0}$ are isomorphic to:
%   \[
%   \PP(2,2) \times \cMorb
%   \]
%   where $\cMorb$ denotes the moduli space of smooth genus-zero
%   orbifold curves with four stacky marked points such that the
%   orbifold structure near each marked point is $[\CC/\mu_2]$ for the
%   non-trivial action of $\mu_2$ on $\CC$, and the virtual fundamental
%   classes are:
%   \begin{align*}
%     & \big[X_{0,\vec{4},0}\big]^\vir = \be\Big[\big(R^1 \pi_\star \ev^\star
%     \cO(1)\big)^{\oplus2}\Big] \\
%     & \big[Y_{0,\vec{4},0}\big]^\vir = \be\Big[R^1 \pi_\star \ev^\star
%     \cO(1)\Big] 
%   \end{align*}
%   Thus even though the non-equivariant limit of the ``equivariant
%   Euler twist'':
%   \[
%   \be^{S^1} \Big(R^\bullet \pi_\star \ev^\star \cO(1)\Big)
%   \]
%   does not exist, there is a sense in which the non-equivariant limit
%   of:
%   \[
%   [X_{0,\vec{4},0}]^\vir \cap \be^{S^1} \Big(R^\bullet \pi_\star \ev^\star \cO(1)\Big)
%   \]
%   exists and equals $[Y_{0,\vec{4},0}]^\vir$.  The next example is free from
%   this defect.
% \end{remark}

\section{A non-trivial example}

Let $X$ be the orbifold $\PP(1,1,1,2,2,2,2)$, and let $Y =
\PP(1,1,2,2,2)$ be the orbifold complete intersection in $X$ defined
by the vanishing of a section of $\cO(1)\oplus \cO(2)$.  Let
$X_{0,\vec{4},0}$ and $Y_{0,\vec{4},0}$ denote the moduli stacks of
genus-zero degree-zero stable maps to (respectively) $X$ and $Y$, from
orbicurves with four marked points such that the isotropy group at
each marked point is $\mu_2$.  Let $i:Y \to X$ be the inclusion map
and $\iota:Y_{0,\vec{4},0} \to X_{0,\vec{4},0}$ be the induced
morphism of moduli stacks.  We have:
\begin{align*}
  & \vdim X_{0,\vec{4},0} = 1 & \vdim Y_{0,\vec{4},0} = 1
  \intertext{and the coarse moduli spaces are:}
  & |X_{0,\vec{4},0}| = \PP^3 \times \Mbar_{0,4} 
  & |Y_{0,\vec{4},0}| = \PP^2 \times \Mbar_{0,4} 
\end{align*}
where $\Mbar_{0,4}$ is Deligne--Mumford space.  Recall that the
rational homology and cohomology groups of a smooth stack coincide
with the rational homology and cohomology groups of the coarse moduli
space \cite{AGV}*{\S2}.  We therefore regard all virtual fundamental
classes, cohomology classes, Chern classes, etc.~in our calculation as
living on the coarse moduli spaces of the stacks involved.

\begin{proposition}
  \label{pro:modulistacks}
  We have:
  \begin{align*}
    & X_{0,\vec{4},0} = \PP(2,2,2,2) \times \Mbar_{0,4} 
    & Y_{0,\vec{4},0} = \PP(2,2,2) \times \Mbar_{0,4}
  \end{align*}
\end{proposition}

\begin{proof}
  We prove the proposition only for $X_{0,\vec{4},0}$; the argument
  for $Y_{0,\vec{4},0}$ is almost identical.  The moduli stack
  $X_{0,\vec{4},0}$ is a $\mu_2$-gerbe over the coarse moduli space
  $|X_{0,\vec{4},0}|$.  Such gerbes necessarily have trivial
  lien\footnote{The lien of a gerbe is also known as its band.
    For a careful discussion of bands and the classification of
    gerbes, see \cite{Moerdijk}*{Lecture~3}.}, and
  thus are classified by the sheaf cohomology group:
  \begin{align*}
    H^2(|X_{0,\vec{4},0}|,\mu_2)  & \cong 
    H^2(\PP^3,\mu_2) \times
    H^2(\Mbar_{0,4},\mu_2) \\
    & \cong \mu_2 \times \mu_2
  \end{align*}
  The gerbe $\PP(2,2,2,2) \times \Mbar_{0,4}$ over $\PP^3 \times
  \Mbar_{0,4}$ is non-trivial on the first factor and trivial on the
  second factor, and therefore corresponds to the class $({-1},1) \in
  \mu_2 \times \mu_2$.  It thus suffices to show that the gerbe
  $X_{0,\vec{4},0}$ over $\PP^3 \times \Mbar_{0,4}$ also corresponds
  to the class $({-1},1) \in \mu_2 \times \mu_2$.  

  Let $\pi_1$ and $\pi_2$ denote the projections to (respectively) the
  first and second factors of the product $\PP^3 \times \Mbar_{0,4}$.
  There is a commutative diagram:
  \begin{equation}
    \xymatrix{
      \PP(2,2,2,2) \ar[d] & X_{0,\vec{4},0}  \ar[l]_-{\Phi}
      \ar[d] \\ 
      \PP^3 & \PP^3 \times \Mbar_{0,4} 
      \ar[l]_-{\pi_1}
    }
  \end{equation}
  where each vertical arrow is the canonical map from a stack to its
  coarse moduli space, and $\Phi$ is the natural morphism coming from
  the fact that $X_{0,\vec{4},0}$ is a moduli stack of degree-zero
  maps.  This implies that restricting the gerbe $X_{0,\vec{4},0}$
  over $\PP^3 \times \Mbar_{0,4}$ to a fiber of $\pi_2$ yields the
  non-trivial gerbe $\PP(2,2,2,2)$ over $\PP^3$.  On the other hand,
  restricting the gerbe $X_{0,\vec{4},0}$ over $\PP^3 \times
  \Mbar_{0,4}$ to a fiber of $\pi_1$ yields the trivial gerbe
  $(B\mu_2)_{0,\vec{4},0}$ over $\Mbar_{0,4}$.  Thus the gerbe
  $X_{0,\vec{4},0}$ over $\PP^3 \times \Mbar_{0,4}$ corresponds to the
  class $({-1},1) \in \mu_2 \times \mu_2$.  The Proposition is proved.
\end{proof}

We will show that there is no cohomology class $e$ on
$X_{0,\vec{4},0}$ such that:
\[
\iota_\star \big[Y_{0,\vec{4},0}\big]^\vir = \big[X_{0,\vec{4},0}\big]^\vir \cap
e
\]
As before, this amounts to showing that $\iota_\star
[Y_{0,\vec{4},0}]^\vir$ and $[X_{0,\vec{4},0}]^\vir$ are not scalar
multiples of each other.  Consider the universal families over the
moduli stacks $X_{0,\vec{4},0}$ and $Y_{0,\vec{4},0}$:
\begin{align*}
  & \xymatrix{
    \cC_X \ar[r]^\ev \ar[d]_\pi & X \\
    X_{0,\vec{4},0}
  }
  &
  \xymatrix{
    \cC_Y \ar[r]^\ev \ar[d]_\pi & Y \\
    Y_{0,\vec{4},0}
  }
\end{align*}
The moduli stack $X_{0,\vec{4},0}$ is smooth, with obstruction bundle
$\cV^{\oplus 3}$ where:
\[
\mathcal{V} = R^1\pi_\star (\ev^\star \cO_X(1))
\]
Thus the virtual fundamental class of $X_{0,\vec{4},0}$ is:
\begin{equation}
  \label{eq:VFCX}
  \big[X_{0,\vec{4},0}\big]^\vir = \big[X_{0,\vec{4},0}\big] \cap \be(\cV)^3
\end{equation}
The moduli stack $Y_{0,\vec{4},0}$ is also smooth, with obstruction bundle:
\[
\Big[R^1\pi_\star (\ev^\star \cO_Y(1)) \Big]^{\oplus 2} 
\]
and since the universal family over $Y_{0,\vec{4},0}$ is the
restriction to $Y_{0,\vec{4},0}$ of the universal family over
$X_{0,\vec{4},0}$:
\[
\xymatrix{ 
  & \cC_X \ar[rr]^{\ev} \ar'[d]_>>>>{\pi}[dd] && X \\
  \cC_Y \ar[ur] \ar[rr]_<<<<<<<<{\ev} \ar[dd]_{\pi} && Y \ar[ur]_i \\
  & X_{0,\vec{4},0}\\
  Y_{0,\vec{4},0} \ar[ur]_{\iota}
}
\]
it follows that:
\[
\Big[R^1\pi_\star (\ev^\star \cO_Y(1)) \Big]^{\oplus 2} =
\iota^\star \cV^{\oplus 2}
\]
Thus the virtual fundamental class of $Y_{0,\vec{4},0}$ is:
\begin{equation}
  \label{eq:VFCY}
  \big[Y_{0,\vec{4},0}\big]^\vir = \big[Y_{0,\vec{4},0}\big] \cap \be(\iota^\star \cV)^2
\end{equation}

We next identify the Euler class of $\cV$.  As before, let $\pi_1$ and
$\pi_2$ denote the projections to (respectively) the first and second
factors of the coarse moduli space $|X_{0,\vec{4},0}| = \PP^3 \times
\Mbar_{0,4}$.  Let $h \in H^2(|X_{0,\vec{4},0}|)$ be the pullback
along $\pi_1$ of the first Chern class of the line bundle $\cO(1) \to
\PP^3$.  Let $\psi \in H^2(|X_{0,\vec{4},0}|)$ be the pullback along
$\pi_2$ of the universal cotangent line class on $\Mbar_{0,4}$
corresponding to the first marked point.  Note that $\{h, \psi\}$
forms a basis for $H^2(|X_{0,\vec{4},0}|)$.

\begin{lemma}
  \label{lem:identifyV}
  \[
  \be(\cV) =  \textstyle \frac{1}{2}  (h-\psi)
  \]
\end{lemma}

\begin{proof} 
  Consider the universal family:
  \begin{equation}
    \label{eq:diagram2}
    \xymatrix{
      \cC_X \ar[r]^\ev \ar[d]_\pi & X \\
      X_{0,\vec{4},0}
    }
  \end{equation}
  and recall that $X_{0,\vec{4},0} \cong \PP(2,2,2,2) \times
  \Mbar_{0,4}$.  We have:
  \begin{align*}
    \cV &= {-\pi_\star} \ev^\star \big(\cO_X(1)\big) && \text{(K-theory pushforward)}
  \\
  &= {-\pi_\star} \pi^\star \big(\cO_{\PP(2,2,2,2)}(1)\big) \\
  &= \cO_{\PP(2,2,2,2)}(1) \boxtimes \big({-\pi_\star \cO_{\cC_X}}\big)
  && \text{(projection formula)} 
\end{align*}
We saw in the proof of Proposition~\ref{pro:modulistacks} that
restricting the gerbe $X_{0,\vec{4},0}$ over $\PP^3 \times
\Mbar_{0,4}$ to a fiber of $\pi_2$ yields $\PP(2,2,2,2)$.  The
restriction of $\cV$ to this copy of $\PP(2,2,2,2)$ is
$\cO_{\PP(2,2,2,2)}(1)$, and so:
\[
\be(\cV) = \textstyle \frac{1}{2} h + \alpha \psi
\]
for some scalar $\alpha$.  We saw in the proof of
Proposition~\ref{pro:modulistacks} that restricting the gerbe
$X_{0,\vec{4},0}$ over $\PP^3 \times \Mbar_{0,4}$ to a fiber of
$\pi_1$ yields $(B\mu_2)_{0,\vec{4},0}$.  The restriction of $\cV$ to
this copy of $(B\mu_2)_{0,\vec{4},0}$ is $\EE^\vee$, where $\EE$ is
the Hodge bundle on $(B\mu_2)_{0,\vec{4},0}$, and so we can determine
the scalar $\alpha$ by comparing the integrals:
\begin{align*}
  \int_{(B\mu_2)_{0,\vec{4},0}} c_1\big(\EE\big) = {\frac{1}{4}}
  &&
  \int_{\Mbar_{0,4}} \psi_1 = 1
\end{align*}
The right-hand integral here is well-known; the left-hand integral is
computed in \cite{Johnson--Pandharipande--Tseng}*{\S3.1}.
\end{proof}
 
\begin{proposition}
  The classes $\iota_\star [Y_{0,\vec{4},0}]^\vir$ and $[X_{0,\vec{4},0}]^\vir$ are
  not scalar multiples of each other.
\end{proposition}

\begin{proof}
 Combining \eqref{eq:VFCX} and \eqref{eq:VFCY} with
  Lemma~\ref{lem:identifyV}, we have:
  \begin{align*}
    [X_{0,\vec{4},0}]^\vir & =  [X_{0,\vec{4},0}] \cap \Big( \textstyle \frac{1}{8} h^3 - \frac{3}{8}
    h^2 \psi \Big)
    \intertext{and:}
    \iota_\star [Y_{0,\vec{4},0}]^\vir 
    &= \iota_\star [Y_{0,\vec{4},0}] \cap \be(\cV)^2  \\
    &=  [X_{0,\vec{4},0}] \cap \Big( h \cup \textstyle \frac{1}{4}(h-\psi)^2 \Big)
    \\
    &= [X_{0,\vec{4},0}] \cap \Big( \textstyle \frac{1}{4} h^3 - \frac{1}{2}
    h^2 \psi \Big)
  \end{align*}
  Since $h^3$ and $h^2 \psi$ are linearly independent in
  $H^6\big(|X_{0,\vec{4},0}|\big)$, the Proposition follows.
\end{proof}

\section{Convexity}
\label{sec:convexity}

Our examples show that the key property underlying ($\dagger$) is not
positivity ($\ast$) of $E$ but rather convexity of $E$.  Recall that a
vector bundle $E \to X$ is called convex if and only if $H^1(C,f^\star
E) = 0$ for all stable maps $f:C \to E$ from genus-zero (orbi)curves.
Suppose that $E = \oplus_j E_j$ is a direct sum of line bundles and
that each line bundle $E_j$ satisfies ($\ast$).  If $X$ is a smooth
variety then $E$ is automatically convex but, as we will discuss
below, this need not be the case if $X$ is an orbifold.

Let $X$ be a smooth projective variety or smooth orbifold, and let $E
\to X$ be a convex vector bundle.  Let:
\[
\xymatrix{
  C \ar[d]_{\pi} \ar[r]^{\ev} & X \\
  X_{0,n,d}}
\]
be the universal family over the moduli stack $X_{0,n,d}$ of
genus-zero stable maps and let $E_{0,n,d} = R^0 \pi_\star \ev^\star
E$.  Convexity implies that $R^1 \pi_\star \ev^\star E = 0$, and hence
that $E_{0,n,d}$ is a vector bundle on $X_{0,n,d}$. 

\begin{proposition}[Convexity implies ($\dagger$)]
  Let $X$ be a smooth projective variety or orbifold, let $E \to X$ be
  a convex vector bundle, and let $Y$ be the subvariety or
  suborbifold of $X$ cut out by a generic section $s$ of $E$.  Let
  $i:Y \to X$ be the inclusion map, and let $\iota:Y_{0,n,\delta} \to
  X_{0,n,i_\star \delta}$ be the induced morphism of moduli stacks.  Then:
  \[
  \sum_{\delta : i_\star \delta = d}
  \iota_\star [Y_{0,n,\delta}]^\vir = [X_{0,n,d}]^\vir \cap \be(E_{0,n,d})
  \]
\end{proposition}

\begin{proof}
  The stacks $X_{0,n,d}$ and $Y_{0,n,\delta}$ carry perfect obstruction
  theories relative to the Artin stack $\Mfr$ of marked twisted curves
  \cite{AGV}:
  \begin{equation}
    \label{eq:obstructiontheories}
    \begin{aligned}
      &\text{$\big(R^\bullet \pi_\star \ev^\star TX\big)^\vee$ for
        $X_{0,n,d}$} \\
      &\text{$\big(R^\bullet \pi_\star \ev^\star TY\big)^\vee$ for
        $Y_{0,n,\delta}$} 
    \end{aligned}
  \end{equation}
  Write:
  \[
  Y_d = \coprod_{\delta:i_\star \delta=d} Y_{0,n,\delta}
  \]
  and consider the 2-Cartesian digram of Deligne--Mumford stacks:
  \[
  \xymatrix{Y_d \ar[r]^-\iota \ar[d]_\iota & X_{0,n,d} \ar[d]^{\tilde{s}} \\ X_{0,n,d}
    \ar[r]^0 & E_{0,n,d}}
  \]
  where $0$ is the zero section of $E_{0,n,d}$ and $\tilde{s}$ is the
  section of $E_{0,n,d}$ induced by $s$.  For a morphism $A\to B$ of
  stacks, let $L_{A/B}$ denote the relative cotangent complex
  \cite{Laumon--Moret-Bailly}.  There is a morphism of distinguished
  triangles in the derived category of sheaves on $Y_d$:
  \[
  \xymatrix{
    \iota^\star(R\pi_\star\ev^\star T_X)^\vee \ar[r] \ar[d] & (R\pi_\star\ev^\star T_Y)^\vee \ar[r] \ar[d]& \iota^\star E^\vee_{0,n,d}[1] \ar[r] \ar[d]& \iota^\star(R\pi_\star\ev^\star T_X)^\vee[1] \ar[d]\\
    \iota^\star L_{X_{0,n,d}/\Mfr} \ar[r] & L_{Y_d/\Mfr} \ar[r] & L_{Y_d/X_{0,n,d}} \ar[r] &\iota^\star L_{X_{0,n,d}/\Mfr}[1]}
  \]
  and, since $E$ is convex, we have:
  \[
  \iota^\star E^\vee_{0,n,d}[1] = \iota^\star L_{X_{0,n,d}/E_{0,n,d}}
  \]
  Thus the perfect obstruction theories \eqref{eq:obstructiontheories}
  are compatible over $\iota:Y_d \to X_{0,n,d}$ in the sense of
  Behrend--Fantechi
  \cite{Behrend--Fantechi}*{Definition~5.8}. Functoriality for the
  virtual fundamental class \cite{Kim--Kresch--Pantev} now implies
  that:
  \[
  0^![X_{0,n,d}]^\vir=\sum_{\delta:i_\star \delta = d} [Y_{0,n,\delta}]^\vir
  \]
  The Proposition follows.
\end{proof}

\begin{remark}
  In the non-convex case, much of this goes through but the perfect
  obstruction theories involved are no longer compatible along
  $\iota$.
\end{remark}

\begin{remark}
  Suppose now that $X$ is a smooth orbifold and that $E \to X$ is a
  line bundle on $X$ that satisfies ($\ast$).  A straightforward
  argument involving orbifold Riemann--Roch \cite{AGV}*{\S7} shows
  that $E$ is convex if and only if $E$ is the pullback of a line
  bundle on the coarse moduli space of $X$.
\end{remark}

\section{Conclusion}
\label{sec:conclusion}

We have seen that the Quantum Lefschetz Hyperplane Principle can fail
for orbifold complete intersections, in cases where the bundle
defining the complete intersection is non-convex.  Thus at the moment
we lack tools to prove mirror theorems for such complete
intersections, even when the ambient orbifold is toric.  A positivity
condition alone ($\ast$) is not enough to force convexity: it is
necessary also for the bundle involved to be the pullback of a bundle
on the coarse moduli space.  This latter condition is very
restrictive, and so ``most'' bundles on orbifolds are not convex.

Despite the examples in this paper one may still hope that, under some
mild conditions, genus-zero Gromov--Witten invariants of orbifold
complete intersections coincide with appropriate twisted
Gromov--Witten invariants.  For example, the equivariant-Euler twisted
I-function $I^{\mathrm{tw}}(t,z)$ in \cite{CCIT}*{Theorem 4.8} admits
a non-equivariant limit when the bundle $E$ and the parameter $t$
involved satisfy certain mild conditions \cite{CCIT}*{Corollary 5.1}.
This is surprising, because the conditions there do not imply
convexity. So one can hope that the twisted I-function still
calculates the genuine invariants in such cases.  (In the examples in
this paper, the relevant twisted I-function does not admit a
non-equivariant limit.)  For example, Guest--Sakai computed the small
quantum cohomology of a degree 3 hypersurface in P(1,1,1,2) from the
differential equation satisfied by the twisted I-function
\cite{Guest--Sakai}, showing that the result coincides with Corti's
geometric calculation.  

Establishing the relationship between Gromov--Witten invariants of
orbifold complete intersections and twisted Gromov--Witten invariants
will require new methods.  In the case of positive, non-convex bundles
on orbifolds, the geometry involved is very similar to that which
occurs when studying higher-genus stable maps to hypersurfaces in
smooth varieties.  Zinger and his coauthors
\citelist{\cite{Li--Zinger} \cite{Vakil--Zinger}} and Chang--Li
\cite{Chang--Li} have made significant progress in this area recently,
and it will be interesting to see if their techniques shed light on
the genus-zero orbifold case too.

\bibliographystyle{amsplain}

\end{document}